\documentclass[12pt,twoside]{amsart}
\usepackage{amssymb}
\usepackage{amsmath}

\usepackage{graphics}

\input arrow.tex

\newtheorem{theorem}{Theorem}[section]
\newtheorem{lemma}[theorem]{Lemma}
\newtheorem{proposition}[theorem]{Proposition}
\newtheorem{corollary}[theorem]{Corollary}
\theoremstyle{remark}
\newtheorem{remark}[theorem]{Remark}
\newtheorem{remarks}[theorem]{Remarks}
\newtheorem{example}[theorem]{Example}
\theoremstyle{definition}
\newtheorem{definition}[theorem]{Definition}

\newcommand{\oo}{\mathcal O}
\newcommand{\ii}{\mathcal I}

\newcommand{\pp}{\mathbb P}
\newcommand{\zz}{\mathbb Z}

\newcommand{\proj}{\operatorname{Proj}}
\newcommand{\Hom}{\operatorname{Hom}}
\newcommand{\rk}{\operatorname{rank}}
\newcommand{\coker}{\operatorname{coker}}
\newcommand{\depth}{\operatorname{depth}}

\newcommand{\ext}{\operatorname{Ext}}
\newcommand{\pd}{\operatorname{pd}}
\newcommand{\spec}{\operatorname{Spec}}
\newcommand{\Ass}{\operatorname{Ass}}

\title[Torsion--free sheaves and ACM schemes]{Torsion--free sheaves and ACM schemes}

\author{S. Greco}
\author{R. Notari}
\author{M.L. Spreafico}

\address{S. Greco, Dipartimento di Matematica, Politecnico
di Torino, I-10129 Torino, Italia} \email{silvio.greco@polito.it}

\address{Roberto Notari, Dipartimento di Matematica ``Francesco
 Brioschi", Politecnico di Milano, I-20133 Milano, Italia}
\email{roberto.notari@polimi.it}

\address{M.L. Spreafico, Dipartimento di Matematica, Politecnico
di Torino, I-10129 Torino, Italia}
\email{maria.spreafico@polito.it}

\date{\today}

\begin{document}

\begin{abstract} In this paper we study short exact sequences $ 0
\to \mathcal P \to \mathcal N \to \ii_D(k) \to 0 $ with $ \mathcal
P, \mathcal N $ torsion--free sheaves and $ D $ closed projective
scheme. This is a classical way to construct and study projective
schemes (e.g. see \cite{hart-1974}, \cite{hart-2}, \cite{mdp},
\cite{serre-1960}). In particular, we give homological conditions
on $ \mathcal P $ and $ \mathcal N $ that force $ D $ to be ACM,
without constrains on its codimension. As last result, we prove
that if $ \mathcal N $ is a higher syzygy sheaf of an ACM scheme $
X,$ the scheme $ D $ we get contains $ X.$
\end{abstract}

\subjclass[2000]{14M05, 14F05}

\keywords{Resolutions, Small projective dimension, Construction of
schemes, Syzygy sheaves}
\thanks{All the authors have been supported by the framework of Prin 2008 ``Geometria delle variet\`{a} algebriche e dei loro spazi di moduli", cofinanced by MIUR}

\maketitle


\section{Introduction}

Homological methods have proved to be very useful in studying
projective schemes. For example, many information on the geometry
of a closed scheme $ X \subseteq \pp^r $ are encoded in the
minimal free resolution of the saturated ideal $ I_X $ of $ X.$
Homological methods are used also to construct schemes with
prescribed properties. For example, in \cite{mdp}, M.
Martin--Deschamps and D. Perrin gave a homological construction of
the ideal of a curve $ C $ in $ \pp^3 $ with a prescribed
Hartshorne--Rao module and of minimal degree. In more details,
given a graded Artinian $ R:=K[x,y,z,w]$--module $ M $ with
minimal free resolution $$ 0 \to L_4 \to L_3 \to L_2 \to L_1 \to
L_0 \to M \to 0,$$ they show how to compute a free graded $
R$--module $ P $ such that the cokernel of a general injective map
$ \gamma: P \to N := \ker(L_1 \to L_0) $ is isomorphic to the
saturated ideal of a locally Cohen--Macaulay curve $ C \subset
\pp^3,$ up to a shift in grading, that is to say, they produce a
short exact sequence
\begin{equation} \label{PNI} 0 \to P \stackrel{\gamma}{\longrightarrow} N \to I_C(k) \to
0.\end{equation}

An analogous sequence was used first by J.P.Serre in
\cite{serre-1960} to construct subcanonical curves in $ \pp^3.$ To
this end, he considered a rank $ 2 $ vector bundle $ \mathcal N, $
a global section $ s $ whose zero--set has codimension $ 2,$ and
the corresponding map $ \oo \stackrel{s}{\to} \mathcal N. $ The
image of the dual map $ \mathcal N^\vee \to \oo $ is the ideal
sheaf of a subcanonical curve $ C \subset \pp^3.$ J.P. Serre's
construction was generalized to construct codimension $ 2 $
schemes in $ \pp^r $ (see \cite{hart-1974}, \cite{okonek}, among
others) and to sections, whose zero--set has codimension $ 2,$ of
reflexive rank $ 2 $ sheaves on $ \pp^3 $ by R. Hartshorne (see
\cite{hart-2}). In the new more general setting, the constructed
curves were generically locally complete intersection curves.

While studying the construction of minimal curves by M.
Martin--Deschamps and D. Perrin given in \cite{mdp}, we applied it
to syzygy modules of $ 0$--dimensional schemes of $ \pp^3 $
instead of syzygy modules of graded Artinian $ R$--modules. The
curves we produced were all arithmetically Cohen--Macaulay. To
understand why the curves share this unexpected property, we were
led to consider all the previous apparently different
constructions from the same point of view, getting as result a
quite general construction of arithmetically Cohen--Macaulay
schemes of arbitrary codimension. For particular choices, we
construct arithmetically Cohen--Macaulay schemes containing a
given scheme with the same property but of larger codimension.

We outline the structure of the paper. In section 2, first of all
we describe some properties of torsion--free coherent sheaves, and
their cohomology. Then, we get some bounds on the projective
dimensions of $ \mathcal N $ and $ \mathcal P $ in terms of the
codimension of $ D $ and of the cohomology of its ideal sheaf $
\mathcal I_D. $ Finally, we recall the well known result of
Martin--Deschamps and Perrin, described in \cite{mdp}, about
maximal subsheaves which allows us to assure that the cokernel of
a given injective map $ \mathcal P \to \mathcal N $ is an ideal
sheaf.

Section 3 is the heart of the paper. At first, we give some
conditions on the coherent torsion--free sheaves $ \mathcal N $
and $ \mathcal P $ to assure that the short exact sequence
(\ref{PNI}) ends with the ideal sheaf of a closed arithmetically
Cohen--Macaulay subscheme $ D $ of $ \pp^r $ of codimension $
2+\pd(\mathcal P),$ where $ \pd(\mathcal P) $ is the projective
dimension of $ \mathcal P.$ Moreover, we show that the
construction characterizes the couple $ (D, \mathcal P) $ in the
sense that starting from an arithmetically Cohen--Macaulay scheme
$ D $ and a torsion--free coherent sheaf $ \mathcal P,$ we can
construct a sheaf $ \mathcal N $ fulfilling our conditions.

In the codimension $ 2 $ case we give a geometrical description of
our construction associating to any non--zero element of $ H^0(D,
\omega_D(c)) $ an extension as (\ref{PNI}). This is a new reading
of the analogous result of  \cite{serre-1960}, for coherent
torsion--free sheaves, without bounds on the rank of $ \mathcal N.
$ We show also that some schemes we obtain in our setting cannot
be obtained with Hartshorne's construction, and conversely. So,
the two constructions are not the same one.

Section 4 is devoted to solve the problem of finding a codimension
$ s $ closed scheme $ D $ containing a given codimension $ t (>s)
$ scheme $ X,$ them both arithmetically Cohen--Macaulay. We end
the section with some examples.


\section{Preliminary results}

Let $ K $ be an algebraically closed field, and let $ R = K[x_0,
\dots, x_r] $ be the graded polynomial ring. Let $ \pp^r =
\proj(R) $ be the projective space of dimension $ r $ over $ K.$
If $ X \subseteq \pp^r $ is a closed scheme, we denote by $ \ii_X
$ its ideal sheaf in $ \oo_{\pp^r} $ and by $ I_X $ its saturated
ideal in $ R,$ and it holds $ I_X = H^0_*(\pp^r, \ii_X).$

By $R$--module we mean ``graded $ R$--module". By sheaf we mean
``coherent $ \oo_{\pp^r}$--module". If $ F $ is a $ R$--module we
denote by $ \mathcal F $ the corresponding sheaf, namely $
\mathcal F := \tilde F.$

We recall that a local ring $ A $ is Cohen--Macaulay if $ \dim(A)
= \depth(A).$ A ring $ A $ is Cohen--Macaulay if $
A_{\mathfrak{M}} $ is Cohen--Macaulay for every maximal ideal $
\mathfrak{M} \subset A.$ A scheme $ X $ is Cohen--Macaulay if the
ring $ \oo_{X,x} $ is Cohen--Macaulay for every closed point $ x
\in X.$  A closed scheme $ X \subseteq \pp^r$ is arithmetically
Cohen--Macaulay (ACM, for brief) if the coordinate ring $ R_X =
R/I_X $ is a Cohen--Macaulay ring. This is equivalent to say that
$H^i_*(\ii_X) = 0$ for $ 1 \le i \le \dim(X).$

For any finitely generated $ R$--module $ P $ we denote by $
\pd(P) $ the projective dimension of $ P,$ that is to say, the
length of the minimal free resolution of $ P $ (\cite{ei}, Theorem
19.1 and the previous Definition).

Let $ D \subseteq \pp^r $ be a closed scheme, and let $ I_D
\subseteq R $ be its saturated ideal. If $$ 0 \to F_t \to \dots
\to F_2 \to F_1 \to I_D \to 0 $$ is the minimal free resolution of
$ I_D,$ with $ t \leq r,$ and $ P $ is the kernel of $ F_1 \to
I_D,$ then we have a short exact sequence $$ 0 \to P \to F_1 \to
I_D \to 0 $$ which is equivalent to the minimal free resolution.
The $ R$--module $ P $ is a torsion--free finitely generated $
R$--module with projective dimension $ \pd(P) = \pd(I_D)-1.$ We
can also consider the short exact sequence $$ 0 \to \mathcal P \to
\mathcal F_1 \to \ii_D \to 0 $$ obtained by considering the
sheaves associated to the modules in the former sequence. Of
course, $ \mathcal P $ is a torsion--free sheaf, and $ \mathcal
F_1 $ is dissoci\'{e}, according to the following definitions.

\begin{definition} A $ R$--module $ M $ is torsion--free if every
non--zero element of $ R $ is a non zero--divisor of $M.$

A sheaf $ \mathcal F $ on $ \pp^r $ is torsion--free if $ \mathcal
F(U) $ is a torsion--free $ \oo_{\pp^r}(U)$--module for every open
subset $ U \subseteq \pp^r.$ This is equivalent to say that
$\mathcal F_x$ is torsion free over $\oo_{\pp^r ,x}$ for every $x
\in \pp^r$.
\end{definition}

\begin{definition} Let $ \mathcal F $ be a sheaf on $ \pp^r.$ We
say that $ \mathcal F $ is dissoci\'{e} of rank $ s $ if $$
\mathcal F = \oplus_{i=1}^s \oo_{\pp^r}(a_i) $$ for suitable
integers $ a_1, \dots, a_s.$
\end{definition}

Of course, if $ F $ is a free $ R$--module, then $ \mathcal F =
\tilde F $ is dissoci\'{e}. Conversely, if $ \mathcal F $ is
dissoci\'{e}, then $ H^0_*(\mathcal F) $ is a free $ R$--module.

Generalizations of the approach consist in relaxing the strong
hypothesis ``dissoci\'{e}" on $ \mathcal F_1.$ Hence, let us
consider the short exact sequence
\begin{equation} \label{pnd} 0 \to \mathcal P
\stackrel{\gamma}{\longrightarrow} \mathcal N \to \ii_D(k) \to 0
\end{equation} with $ \mathcal P $ torsion--free, and $ k \in
\zz.$ Standard arguments allow us to prove that $ \mathcal N $ is
torsion--free, as well. So, the weakest hypothesis on $ \mathcal N
$ is torsion--free. On the other hand, short exact sequences are
classified by $ \ext^1_R(\ii_D, \mathcal P).$

As we are interested in sequences of sheaves, it will help to have
the analogue for sheaves of the minimal free resolution and of
projective dimension of a graded finitely generated module.

By (\cite{hart}, Ch. II, Corollary 5.18), we have that any sheaf $
\mathcal P $ admits a {\em dissoci\'{e} resolution}, namely a
resolution by dissoci\'{e} sheaves. We need to be more precise on
this point, and so we begin with some preliminaries.

\begin{remark} \label{ass} \rm
We recall some facts about associated points. For more details see
e.g (\cite{matsumura}, Ch. 3), where the case of (ungraded)
modules is dealt with. Extending to sheaves is straightforward.

(i) Let $\mathcal F$ be a sheaf. A (not necessarily closed) point
$y \in \pp^r$ is associated to $\mathcal F$ if there is an open
affine $U = \spec(A) \subseteq \pp^r$ containing $y$ such that the
prime ideal of $A$ corresponding to $y$ is associated to the
$A$-module $\Gamma (U,\mathcal F)$; this is equivalent to say that
$\depth_{\oo_{\pp^r},_y} (\mathcal F_y) = 0$.

(ii) The set $Ass(\mathcal F)$ of the associated points to
$\mathcal F$ is finite.

(iii) Any form $f$ of degree $n$ avoiding all elements of
$Ass(\mathcal F)$ induces by multiplication an injective morphism
$ \mathcal F \stackrel{\cdot f}{\longrightarrow} \mathcal F(n)$.
Hence a general form of degree $n$ has this property.

(iv) (\cite{bs}, Exercise 20.4.21) The graded $R$- module
$H^0_*(\mathcal F) $ is finitely generated if and only if
$\Ass(\mathcal F)$ contains no closed points, if and only if
$\depth_{\oo_{\pp^r},_x} (\mathcal F_x) > 0$ for every (closed) $x
\in \pp^r$.
\end{remark}

Now, we can prove that every sheaf has a dissoci\'{e} resolution
of finite length.

\begin{lemma} \label{diss res} Let $\mathcal F$ be a sheaf and
let $M$ be a graded submodule of $H^0_*(\mathcal F)$. Then
\begin{enumerate}
\item any general linear form induces by multiplication an
injective map $M \to M(1)$; \item if M is finitely generated then
$pd(M) \le r$; \item $\mathcal F$ admits a dissoci\'{e} resolution
of length $\le r$.
\end{enumerate}
\end{lemma}

\begin{proof} (1) follows easily from Remark \ref{ass}(iii).

(2) By (1) we have $\depth(M) \ge 1$ and the conclusion follows by
the Auslander-Buchsbaum formula (\cite{ei}, Exercise 19.8).

(3) Since $\mathcal F$ is coherent we have $\mathcal F= \tilde M$,
where M is a suitable finitely generated graded submodule of
$H^0_*(\mathcal F)$ (\cite {hart}, Ch. II, proof of Theorem 5.19).
The conclusion follows from (2), because we get a dissoci\'{e}
resolution of $ \mathcal F $ by sheafifying the minimal free
resolution of $ M.$
\end{proof}

Following (\cite{gr}, Section 2), we define the minimal
dissoci\'{e} resolution of a coherent sheaf.

\begin{definition} \label{min res} Let $ \mathcal
P $ be a sheaf such that $ P : = H^0_*(\mathcal P)$ is finitely
generated. Let $$ 0 \to H_d \to \dots \to H_0 \to  P \to 0
$$ be the minimal free resolution of the $ R$--module P .
We' ll call {\em minimal dissoci\'{e} resolution} of $ \mathcal P
$ the exact sequence $$ 0 \to \mathcal{H}_d \to \dots \to
\mathcal{H}_0 \to \mathcal P \to 0 $$ obtained by sheafifying the
minimal free resolution of $ P.$ (Recall that $\tilde P = \mathcal
P$ by (\cite{hart}, Ch. II, Proposition 5.4)).

Moreover, we define the {\em projective dimension} of $ \mathcal P
$ as $ \pd(\mathcal P) := \pd(P).$
\end{definition}

\begin{remark} \rm It is known that there exist many submodules of
$ P = H^0_*(\mathcal P) $ whose associated sheaf is $ \mathcal P:$
in fact, it is enough that such a submodule $ M $ agrees with $ P
$ for some large degree on. Of course, the sheafification of a
minimal free resolution of $ M $ is still a dissoci\'{e}
resolution of $ \mathcal P,$ and no map is split. However, the
resolution of $ M $ is longer than the minimal one. In fact, from
the short exact sequence of modules $$ 0 \to M \to P \to P/M \to
0,$$ we get that $ \pd(M) = r,$ because $ P/M $ is an Artinian
module. Hence, $ \pd(M) \geq \pd(P),$ as we claimed.
\end{remark}

\begin{remarks} \label{rem-pd} \rm (i) Clearly $\pd( \mathcal P) = 0$
if and only if $ \mathcal P $
dissoci\'{e}.

\noindent (ii) $\pd(\mathcal P) \le r$ whenever defined (Lemma
\ref{diss res}(2) applied with $M= H^0_*(\mathcal  P)$).
\end{remarks}

The next Lemma gives a bound for the projective dimension of a
torsion-free sheaf.

\begin{lemma} \label{initial lemma} Let $ \mathcal P $ be a torsion-free
coherent sheaf on $ \pp^r,$ and let $ P= H^0_*(\mathcal P).$ Then:
\begin{enumerate}
\item $ P $ is finitely generated; \item $P$ torsion free; \item $
\mathcal P $ is a subsheaf of a coherent dissoci\'{e} sheaf; \item
$\pd(P) = \pd(\mathcal P) \leq r-1$.
\end{enumerate}
\end{lemma}

\begin{proof}
(1) It follows easily from Remark \ref{ass}(iv).

(2) Whatever non zero form $f \in R$ of degree $n$  induces, by
multiplication, an injective morphism $ \mathcal P \stackrel{\cdot
f}{\longrightarrow} \mathcal P (n), $ and consequently an
injective homomorphism $ P \stackrel{\cdot f}{\longrightarrow}
P(n).$

(3) By (1) and (2),  $ P $ is a torsion--free $R$-module and hence
it is a graded submodule of a free $R$--module $L$.  Then, $
\mathcal P = \tilde P $ is a subsheaf of $\mathcal L:= \tilde L $
and the claim follows.

(4) By (3) there are a sheaf $\mathcal F$ and an exact sequence $0
\to \mathcal P \to \mathcal L \to \mathcal F \to 0$, whence an
exact sequence of $R$-modules:
$$0\to  P \to L \to M \to 0,$$
where $M$ is a graded submodule of $H^0_*(\mathcal F)$. By Lemma
\ref{diss res}(2) we have $\pd(M) \le r$, whence $\pd(P) \le r-1$.
Now by (1), Definition \ref{min res} applies and the proof is
complete
\end{proof}

From now on, every $ R$--module will be finitely generated, and so
we shall skip this assumption.

It is possible to describe the cohomology of a coherent sheaf, as
we said before.
\begin{lemma} \label{lemma1} Let $ r \geq 3,$
let $ P $ be a $ R$--module and let $ \mathcal P = \tilde P $ be
its associated sheaf. Suppose $ d = \pd(P) < r. $  Then:
\begin{enumerate}
\item $ H^0_*(\mathcal P)=P;$ \item $ H^i_*(\mathcal P)=0 $ for $
1 \leq i \leq  r-d-1;$ \item $ H^{r-d}_*(\mathcal P) \neq 0. $
\item If $ \mathcal P$ is any torsion-free sheaf with $d:=
\pd(\mathcal P),$ then (2) and (3) hold.
\end{enumerate}
\end{lemma}

\begin{proof} We prove claims (1) (2) (3) together,  by induction on $ d.$

If $ d=0,$ the sheaf $ \mathcal P $ is dissoci\'{e} and the claims
hold by (\cite{hart}, Ch. III, Theorem 5.1).

If $d=1$ we have a non--split exact sequence
$$0\to L_1 \to L_0 \to P \to 0$$
with $ L_1 $ and $ L_0 $ free. By passing to sheaves, we get a
non--split exact sequence
\begin{equation}\label{d=1}0\to \mathcal L_1 \to \mathcal L_0
\to \mathcal P \to 0,\end{equation} whence $ \ext^1( \mathcal P,
\mathcal L_1) \neq 0$. It follows easily that $ \ext^1(\mathcal P,
\oo_{\pp^r}(k)) \neq 0 $ for some $ k \in \zz$. On the other hand
by duality and properties of Ext we get $ H^{r-1}(\pp^r, \mathcal
P(-k-r-1)) \cong \ext^1( \mathcal P(-k-r-1), \omega_{\pp^r})
 \cong \ext^1(\mathcal P , \oo_{\pp^r} (k)),$ whence (3).
Since (1) and (2) are immediate from the exact sequence
(\ref{d=1}), the statement holds for $d=1$ as well.

Assume now $d\ge 2$. We have an exact sequence
$$ 0 \to P_1 \to G \to  P \to 0 $$
where $ G $ is a free $R-$module and $ P_1 $ is a $ R$--module
with $ \pd(P_1)=d-1.$ In fact, it is enough to consider the first
short exact sequence that can be obtained from the minimal free
resolution of $ P,$ as explained before. By taking the sheaves
associated to each item, we get the short exact sequence of
sheaves
\begin{equation} \label{p1} 0 \to \mathcal P_1 \to \mathcal G \to
\mathcal P \to 0.
\end{equation}
By induction, we may assume that $ H^0_*(\mathcal P_1)=P_1,$ $
H^i_*(\mathcal P_1)=0 $ for $ 1 \leq i \leq r-(d-1)-1=r-d,$ and
that $ H^{r-d+1}_*(\mathcal P_1) \not= 0.$

By assumption, $ d < r,$ and so $ r-d \geq 1.$ In particular, $
H^1_*(\mathcal P_1)=0. $

By taking the cohomology sequence associated to (\ref{p1}) and
using the assumptions on the cohomology  of $ \mathcal P_1 $ we
get the conclusion.

To prove (4) set $P:= H^0_*(\mathcal P)$. Then by definition and
by Lemma \ref{initial lemma} we have $d = \pd(P)<r$. Since $\tilde
P= \mathcal P$ the conclusion follows by (2) and (3).
\end{proof}

The previous Lemma allows us to generalize Horrocks' splitting
criterion (\cite{okonek}, Theorem 2.3.1) to torsion--free sheaves,
with a completely different proof.
\begin{corollary} A torsion--free sheaf $ \mathcal P $ over
$ \pp^r $ is dissoci\'{e} precisely when $ H^i_*(\mathcal P) = 0 $
for $ i = 1, \dots, r-1.$
\end{corollary}

\begin{proof} Assume $ H^i_*(\mathcal P) = 0$
for $ i = 1, \dots, r-1,$ and set $d:= \pd(\mathcal P).$ By Lemma
\ref{lemma1} (4) we have $ H^i_*(\mathcal P) = 0 $ for $ i = 1,
\dots, r-1-d $ and $ H^{r-d}_*(\mathcal P) \neq 0$. This is
possible only if $d=0$, i.e. if $\mathcal P$ is dissoci\'e. The
converse is clear.
\end{proof}

Now, we consider the short exact sequence (\ref{pnd}). Our first
result relates the codimension of $ D $ and the projective
dimension of $ \mathcal P.$

\begin{proposition} Let $ D \subseteq \pp^r $ be a closed scheme of
codimension $ s,$ with $ s \geq 2,$ and let $ P $ a $ R$--module
with $ \pd(P) = d.$ If $ s-2 > d,$ then $ \ext^j_R(\ii_D(k),
\mathcal{P}) = 0 $ for $ j = 1, \dots, s-2-d. $
\end{proposition}

\begin{proof} We prove the claim by induction on $ d.$

If $ d = 0,$ that is to say $ P $ is a free module, then there
exist $ a_1, \dots, a_n \in \mathbb{Z} $ such that $ P =
\oplus_{i=1}^n R(-a_i).$ By using standard properties of $ \ext $
groups, we have \begin{equation*} \begin{split}
\ext^j&_R(\ii_D(k), \mathcal{P}) = \oplus_{i=1}^n
\ext^j_R(\ii_D(k), \oo_{\pp^r}(-a_i)) = \\ & = \oplus_{i=1}^n
\ext^j_R(\ii_D(k+a_i-r-1), \omega_{\pp^r}) \cong \\ & \cong
\oplus_{i=1}^n H^{r-j}(\pp^r, \ii_D(k+a_i-r-1)) = \\ & =
\oplus_{i=1}^n H^{r-j-1}(D, \oo_D(k+a_i-r-1)) = 0
\end{split}
\end{equation*} as soon as $ r-j-1 > r-s $ by Grothendieck's vanishing
Theorem (\cite{hart}, Ch.III, Theorem 2.7), where $ \omega_{\pp^r}
= \oo_{\pp^r}(-r-1) $ is the canonical sheaf of $ \pp^r.$ Hence, $
\ext^j(\ii_D(k), \mathcal{P}) = 0 $ for $ j = 1, \dots, s-2 $ and
for every $ k \in \mathbb{Z},$ and the claim holds for $ d = 0.$

Assume $ d > 0 $ and the claim to hold for every  $ R$--module
with projective dimension $ d-1.$  As in the proof of Lemma
\ref{lemma1}, we consider the short exact sequence
$$ 0 \to P_1 \to G \to P \to 0 $$ with $ G $ free  and $ P_1 $ of
projective dimension $ d-1.$ By applying $ \Hom(\ii_D(k), -) $ to
the sheafified sequence, we get the exact sequence $$
\ext^i(\ii_D(k), \mathcal{G}) \to \ext^i(\ii_D(k), \mathcal{P})
\to \ext^{i+1}(\ii_D(k), \mathcal{P}_1) \to \ext^{i+1}(\ii_D(k),
\mathcal{G}).$$ From the first part of the proof, we get that $
\ext^i(\ii_D(k), \mathcal{G}) = \ext^{i+1}(\ii_D(k), \mathcal{G})
= 0 $ for every $ k $ and for $ i = 1, \dots, s-3.$ From the
induction assumption, $ \ext^{i+1}(\ii_D(k), \mathcal{P}_1) = 0 $
for every $ k $ and for $ i = 0, \dots, s-2-d.$ Hence, $
\ext^i(\ii_D(k), \mathcal{P}) = 0 $ for every $ k \in \mathbb{Z} $
and for $ i = 1, \dots, s-2-d $ as claimed.
\end{proof}

A direct consequence of the previous Proposition is that we can
predict if $ N $ is the direct sum of $ P $ and $ I_D.$ In fact it
holds:
\begin{corollary}\label{split} Let $ D \subseteq \pp^r $ be a closed scheme of
codimension $ s \geq 2,$ and let $ P $ be a $ R$--module
satisfying $ s-2 > \pd(P).$ Then, the only extension of $ \ii_D(k)
$ with $ \mathcal{P} $ is the trivial one, for every choice of $ k
\in \zz.$ Consequently if there is a non--split exact sequence
(\ref{pnd}), we must have $s \leq \pd(P) + 2.$
\end{corollary}

\begin{proof} The previous Proposition shows that $ \ext^1(\ii_D(k),
\mathcal{P}) = 0 $ and the claim follows.
\end{proof}

Now, we take into account the cohomology of $ D $ to get a bound
on the projective dimension of $ \mathcal N.$
\begin{proposition} \label{pd(N) geq pd(P)+2} Let $ D \subset
\pp^r $ be a closed scheme, and let $ \mathcal{P}, \mathcal{N} $
be torsion--free sheaves such that the short sequence (\ref{pnd})
is exact. If $ \pd(\mathcal N) \geq \pd(\mathcal P) + 2,$ then $
H^{r-\pd(\mathcal N)}_*(\ii_D) \not= 0. $

Conversely, if $ H^j_*(\ii_D) \not= 0 $ for some $ j \in
\mathbb{Z} $ with $ 1 \leq j \leq r - 2 - \pd(\mathcal P),$ then $
\pd(\mathcal N) \geq \pd(\mathcal P) + 2.$
\end{proposition}

\begin{proof}  By Lemma \ref{initial lemma},
we have that $ \pd(\mathcal P) $ and
$ \pd(\mathcal N) $ are strictly smaller than $ r.$   By taking
the long exact cohomology sequence associated to (\ref{pnd}), we
get
$$ H^i_*(\mathcal{P}) \to H^i_*(\mathcal{N}) \to H^i_*(\ii_D) \to
H^{i+1}_*(\mathcal{P}).$$ From Lemma \ref{lemma1}, we know that $
H^j_*(\mathcal{P}) = 0 $ for $ j = 1, \dots, r -\pd(\mathcal P)-
1,$ and so $ H^i_*(\mathcal{N}) \cong H^i_*(\ii_D) $ for $ i = 1,
\dots, r - \pd(\mathcal P) - 2.$

If $ \pd(\mathcal P) + 2 \leq \pd(\mathcal N) < r,$ then $ 1 \leq
r - \pd(\mathcal N) < r - \pd(\mathcal P) - 1.$ Hence, $
H^{r-\pd(\mathcal N)}_*(\mathcal{N}) \cong H^{r-\pd(\mathcal
N)}_*(\ii_D) $ and we get the claim by Lemma \ref{lemma1}.

Assume now that $ H^j(\ii_D(k)) \not= 0 $ for some $ k \in
\mathbb{Z} $ and some $ j $ such that $ 1 \leq j \leq r -
\pd(\mathcal P) - 2.$ Hence, $ H^j_*(\mathcal{N}) \not= 0.$ Again
by Lemma \ref{lemma1}, $ r - \pd(\mathcal N) \leq j $ and so $
\pd(\mathcal N) \geq \pd(\mathcal P) + 2.$
\end{proof}

\begin{remark} \rm In the second part of the previous Proposition, the
hypothesis on $ j $ implies that $ \pd(\mathcal P) \leq r-3.$ This
last inequality is not automatically fulfilled. In fact, let $ D
\subseteq \pp^r $ be a locally Cohen--Macaulay curve with $ H^1_*
\ii_D \not= 0.$ Let $$ 0 \to G_r \to \dots \to G_2 \to G_1 \to I_D
\to 0 $$ be the minimal free resolution of $ I_D $ and let $ P =
\ker(G_1 \to I_D).$ Then, $ \pd(P) = r-2 = \pd(I_D) - 1,$ and $
\pd(G_1) = 0.$ Hence, we cannot apply the previous Proposition to
the short exact sequence $ 0 \to P \to G_1 \to I_D \to 0.$
Nevertheless, it could exist a different short exact sequence $ 0
\to Q \to N \to I_D \to 0 $ with $ \pd(Q) = r-3.$ In this case, $
\pd(N) = r-1.$ Notice that $ r-3 $ is the smallest projective
dimension allowed for the first item of the sequence, because of
the codimension of $ D.$
\end{remark}

\begin{remark} \rm The case considered in the previous
Proposition, namely $ \pd(\mathcal N) \geq \pd(\mathcal P) + 2,$
occurs in the $ \mathcal N$--type resolution of the ideal sheaf of
a locally Cohen--Macaulay curve in $ \pp^3 $ (\cite{mdp}, Ch. II,
Section 4). In that case, $ \mathcal P $ is dissoci\'{e} and $
\pd(\mathcal N) = 2,$ where $ N $ is the second syzygy module of
the Hartshorne--Rao module (graded Artinian $ R$--module) of the
curve, up to a free summand.
\end{remark}

Now, we stress some consequences of the previous Proposition that
we' ll use in next sections.
\begin{corollary} \label{acm and pd} Consider an exact sequence
{\rm (\ref{pnd})} where $D$ has codimension $s \ge 2$. If $D$ is
ACM and the sequence is non-split we have
\begin{equation}\label {pd(N)}
\pd(\mathcal N) \leq \pd(\mathcal P) +1.
\end{equation}
\end{corollary}

\begin{proof} By Corollary \ref{split} the non--splitting
of the sequence (\ref{pnd})
implies that $ s \leq \pd(\mathcal P) + 2.$ If $ \pd(\mathcal N)
\geq \pd(\mathcal P) + 2,$ then $ H^{r-\pd(\mathcal N)}_*(\ii_D)
\not= 0 $ by Proposition \ref{pd(N) geq pd(P)+2}. On the other
hand, $ r-\pd(\mathcal N) \leq r-s $ and so $ H^{r-\pd(\mathcal
N)}_*(\ii_D) = 0 $ because $ D $ is ACM. The contradiction proves
that $ \pd(\mathcal N) \leq \pd(\mathcal P) + 1.$
\end{proof}

\begin{remark} \rm If $ \pd(\mathcal N) \leq \pd(\mathcal P) + 1,$
we can only prove that $ H^i_*(\ii_D) = 0 $ for $ i = 1, \dots,
r-\pd(\mathcal P)-2.$ Hence, $ D $ could not be an ACM scheme if $
s < \pd(\mathcal P) + 2.$
\end{remark}

A further problem related to the sequence (\ref{pnd}) is the
following: given the  modules $ P $ and $ N,$ and an injective map
$ \mathcal P \to \mathcal N,$ when is the cokernel an ideal sheaf?
This problem was considered in \cite{mdp}, and we resume their
results.

At first, we recall the definition and some properties of the
maximal subsheaves, generalizing to $ \pp^r $ the one given for
sheaves on $ \pp^3 $ (\cite{mdp}, Ch. IV, D\'{e}finition 1.1).
\begin{definition} Let $ \mathcal{M} \subset \mathcal{N} $ be
$ \oo_{\pp^r}$--modules. $ \mathcal{M} $ is a maximal subsheaf of
$ \mathcal N $ if for all subsheaves $ \mathcal{M'} \subset
\mathcal{N} $ with $ \rk(\mathcal M) = \rk(\mathcal{M'}) $ such
that $ \mathcal{M} \subseteq \mathcal{M'} \subseteq \mathcal{N},$
we have $ \mathcal{M} = \mathcal{M'}.$
\end{definition}

The interest in such subsheaves lies in the following properties.
\begin{proposition} \label{maximal-sub-sheaf} Let $ \mathcal{M}
\subseteq \mathcal{N} $ be $ \oo_{\pp^r}$--modules. Consider the
following properties:
\begin{enumerate}
\item $ \mathcal{M} $ is maximal; \item $ \mathcal{N}/\mathcal{M}
$ is torsion--free; \item $ \mathcal{N}/\mathcal{M} $ is
torsion--free in codimension $ 1;$ \item $ \mathcal{N}/\mathcal{M}
$ is locally free in codimension $ 1;$ \item $
\mathcal{N}/\mathcal{M} $ has constant rank in codimension $ 1;$
\item $  \mathcal{N}/\mathcal{M} $ is locally a direct summand of
$ \mathcal{N} $ in codimension $ 1.$
\end{enumerate}
Then, $ (1) \Leftrightarrow (2) \Rightarrow (3) \Leftrightarrow
(4) \Leftrightarrow (5) \Rightarrow (6).$ Furthermore, if $
\mathcal{N} $ is torsion--free and $ \mathcal{M} $ is locally
free, they all are equivalent.
\end{proposition}

\begin{proof} The statement was proved for sheaves on $ \pp^3 $
in (\cite{mdp}, Ch. IV, Proposition 1.2),but the proof works
without changes also for sheaves on $ \pp^r.$
\end{proof}

Moreover, in the proof, the authors proved also the existence of
maximal dissoci\'{e} subsheaves of a sheaf $ \mathcal{N}.$

As explained in (\cite{mdp}, Ch.IV, Remark 1.3(c)), in $ \pp^3,$
if $ \mathcal{N} $ is a rank $ n+1 $ vector bundle, and $
\mathcal{M} $ is a rank $ n $ dissoci\'{e} maximal subsheaf of $
\mathcal{N},$ then $ \mathcal{N}/\mathcal{M} $ is a rank $ 1 $
torsion--free sheaf, and so it is an ideal sheaf tensorized times
$ \det(\mathcal{N}) \otimes \det(\mathcal{M}^{-1}).$ Moreover, if
$ \mathcal N $ is not dissoci\'{e}, then the ideal sheaf defines a
curve.


\section{A construction of ACM schemes}

In this section, we consider two coherent torsion--free sheaves $
\mathcal{P} $ and $ \mathcal{N} $ and an injective map $ \gamma:
\mathcal{P} \to \mathcal{N},$ and we study the scheme $ D $ whose
ideal sheaf is isomorphic to $ \coker(\gamma),$ as in \cite{mdp}.
We limit ourselves to consider only the case $ D $ has the largest
codimension to have a non--split exact sequence (\ref{pnd}) (see
Corollary \ref{split}) and $ \mathcal N $ to have the largest
projective dimension to allow $ D $ to be an ACM scheme (see
Corollary \ref{acm and pd}). In more detail, we collect the
hypotheses on $ \mathcal{P} $ and $ \mathcal{N} $ in the following

\begin{equation}
\label{H-hyp} \tag{H}
\begin{split} & (H.1) \ \mathcal{P} \mbox{ is torsion-free and }
s:= \pd(\mathcal P) + 2 \le r;\\ & (H.2) \ \mathcal{N} \mbox{
is torsion-free and } \pd(\mathcal N) \leq \pd(\mathcal P) +1 ; \\
& (H.3) \ \mbox{the polynomial }
\\ & \qquad \quad p(t) :=-\chi(\mathcal N(t-k)) + \chi(\mathcal P(t-k)) +
\binom{t+r}r
\\ & \qquad \mbox{has degree } r-s
\mbox{ for some } k \in \zz. \end{split}
\end{equation}

We remark that, in view of Definition \ref{min res} and Remark
\ref{rem-pd}, the condition about the projective dimensions
required in (H.2) and (H.3) means that $ P := H^0_*(\mathcal P) $
and  $ N := H^0_*(\mathcal N) $ have, respectively, minimal free
resolutions

$$ \ 0 \to G_{s-1}
\stackrel{\Delta_{s-1}}{\longrightarrow} G_{s-2}
\stackrel{\Delta_{s-2}}{\longrightarrow} \dots
\stackrel{\Delta_{2}}{\longrightarrow} G_1 \to P \to 0 $$ and

$$
\ 0 \to F_{s} \stackrel{\delta_{s}}{\longrightarrow} F_{s-1}
\stackrel{\delta_{s-1}}{\longrightarrow} \dots
\stackrel{\delta_{2}}{\longrightarrow} F_1 \to N \to 0.$$

\begin{remarks} \label{rem(H)} \rm  (i) We allow $ F_j = 0 $ for some
$ j $ in the minimal free resolution of $ N.$ In such a case, $
F_{j+h} = 0 $ for every $ h \geq 0.$

(ii) Condition (H.3) implies that $ \rk(\mathcal N) = \rk(\mathcal
P) + 1,$ because the rank of $ \mathcal F $  is equal to $ r! $
times the coefficient of $ t^r $ in $ \chi(\mathcal F(t)). $
Moreover, recalling that $ \oo_{\pp^r}(a) $ has degree $ a $ and
that the degree is additive on exact sequences, we have that $ k =
\deg(\mathcal{N}) - \deg(\mathcal{P}).$
\end{remarks}

Now, we describe the geometric properties of the schemes that can
be obtained from such torsion--free sheaves.
\begin{theorem} \label{NtoD} Let $ \mathcal P $ and $ \mathcal N $
be torsion--free coherent sheaves that fulfil the hypotheses
(\ref{H-hyp}). Assume that there exists an injective map $ \gamma:
\mathcal{P} \to \mathcal{N} $ whose image is a maximal subsheaf of
$ \mathcal{N}$. Then there exists a codimension $ s = 2 +
\pd(\mathcal P) $ scheme $ D,$ closed and ACM, whose ideal sheaf
fits into the short exact sequence (\ref{pnd}) with $ k =
\deg(\mathcal N)-\deg(\mathcal P)$. Moreover, the sequence $0 \to
P \to N \to I_D(k) \to 0$ is exact.
\end{theorem}

\begin{proof} The cokernel of $ \gamma $ is a rank $ 1 $ torsion--free
sheaf $ \mathcal{F}.$ Let $ \mathcal{F}^{\vee\vee} $ be its double
dual. Since $ \mathcal{F} $ is torsion--free, the natural map $
\mathcal{F} \to \mathcal{F}^{\vee\vee} $ is injective. By
(\cite{hart-2}, Corollary 1.2 and Proposition 1.9), $
\mathcal{F}^{\vee\vee} \cong \oo_{\pp^r}(h) $ for some $ h \in
\zz,$ and so $ \mathcal{F} \cong \ii_D(h) \subseteq
\oo_{\pp^r}(h),$ i.e. we have an exact sequence

$$ 0 \to \mathcal{P} \stackrel{\gamma}{\longrightarrow} \mathcal{N}
\to \ii_{D}(h) \to 0.$$

Clearly $ h = \deg(\mathcal N)-\deg(\mathcal P) = k $, and hence
the above sequence coincides with (\ref{pnd}). Now, by Remark
\ref{rem(H)}, (ii), $k$ is the integer occurring in the polynomial
p(t) of (H.3), then $ p(t) $ is the Hilbert polynomial of $D$,
whence $\dim(D) = r-s $ by (H.3). Moreover, (H.2) and the second
part of Proposition \ref{pd(N) geq pd(P)+2} imply that $D$ is ACM.
Finally, by (H1) we have   $\pd(\mathcal P) \le r-2$, whence $r-
\pd(\mathcal P)- 1 \ge 1$. Then Lemma \ref{lemma1}(4) implies that
$H^1_*(\mathcal P)= 0$ and the last statement follows.
\end{proof}

\begin{remark} \label{rem3.3} \rm The map $ \gamma: \mathcal{P} \to \mathcal{N} $
induces a map of complexes between the minimal free resolutions of
$ P $ and $ N.$ Let $ \gamma_i : G_i \to F_i $ be the induced map.
Of course, $ \gamma_i \circ \Delta_{i+1} = \delta_{i+1} \circ
\gamma_{i+1},$ for each $ i \geq 1.$ Hence, a resolution of $
I_D(k) $ can be obtained via mapping cone from (\ref{PNI}), and it
is
\begin{equation}
\label{non-min-res-D} 0 \to \begin{array}{c} G_{s-1} \\
\oplus \\ F_s \end{array}
\stackrel{\varepsilon_s}{\longrightarrow} \begin{array}{c} G_{s-2}
\\ \oplus \\ F_{s-1} \end{array}
\stackrel{\varepsilon_{s-1}}{\longrightarrow} \dots
\stackrel{\varepsilon_3}{\longrightarrow} \begin{array}{c} G_1 \\
\oplus \\ F_2 \end{array}
\stackrel{\varepsilon_2}{\longrightarrow} F_1 \to I_{D}(k) \to 0
\end{equation} where $ \varepsilon_i: G_{i-1} \oplus
F_i \to G_{i-2} \oplus F_{i-1} $ is given by
$$ \left( \begin{array}{cc} \Delta_{i-1} & 0 \\ (-1)^i
\gamma_{i-1} & \delta_i \end{array} \right), \mbox{ for } i \geq
2. $$ We remark that $\varepsilon_2 : G_{1} \oplus F_2 \to F_1 $
is represented by the matrix $ ( \gamma_1, \delta_2).$

For general results on free resolutions, it is clear that the
minimal free resolution of $ I_D(k) $ can be obtained by
cancelling the free modules corresponding to constant non--zero
entries of any matrix representing the map $ \varepsilon_i, i = 2,
\dots, s.$
\end{remark}

\begin{remark} \rm If there exists an injective map $ \gamma: \mathcal{P}
\to \mathcal{N} $ whose image is a maximal subsheaf of $
\mathcal{N} $ of rank $ \rk(\mathcal{P}) = \rk(\mathcal{N}) - 1,$
then the general map in $ \Hom(\mathcal{P}, \mathcal{N}) $ has the
same property.
\end{remark}

Once we have constructed a closed ACM scheme $ D $ of codimension
$ s $ as cokernel of a short exact sequence (\ref{pnd}), we can
construct the minimal free resolution of $ I_D $ and it is $$ 0
\to H_s \stackrel{\sigma_s}{\longrightarrow} \dots
\stackrel{\sigma_2}{\longrightarrow} H_1
\stackrel{\sigma_1}{\longrightarrow} I_D \to 0 $$ where $ H_i =
\oplus_{n \in \zz} R(-n)^{h_i(n)}.$ Let $ K = \ker(\sigma_1).$
Then, the ideal sheaf $ \ii_D $ is also the cokernel of the short
exact sequence
\begin{equation} \label {K to H} 0 \to \mathcal{K} \stackrel{j}{\longrightarrow}
\mathcal{H}_1 \stackrel{\sigma_1}{\longrightarrow} \ii_D \to
0.\end{equation}

Now we compare the two sequences (\ref{pnd}) and (\ref{K to H}).

\begin{proposition} \label{DtoN} Let $D \subseteq \pp^r$ be
an ACM  scheme of codimension $s$ and let
 {\rm (\ref{K to H})} be as above.

 (i) If there is a
 sequence {\rm (\ref{pnd})} with $\pd(\mathcal P) = s-2$ then there
exists a map $ \psi: \mathcal K \to \mathcal P $ such that $
\mathcal N $ is the push--out of $ \mathcal P $ and $ \mathcal
H_1.$

(ii) Conversely, let $
\mathcal P $ be a torsion--free coherent sheaf with $ \pd(\mathcal
P) = s-2.$ Then, for every map $ \psi: \mathcal K \to \mathcal P $
there exists a short exact sequence (\ref{pnd}) whose third item
is $ \ii_D.$
\end{proposition}

\begin{proof} (i) Up to twisting the sequence (\ref{pnd}), we can assume
that $ k = 0.$ The minimal free resolution of $ I_D $ is
$$ 0 \to H_s \stackrel{\sigma_s}{\longrightarrow} \dots
\stackrel{\sigma_2}{\longrightarrow} H_1
\stackrel{\sigma_1}{\longrightarrow} I_D \to 0,$$ and so $
\sigma_1 $ maps the canonical bases of $ H_1 $ onto a minimal set
of generators of $ I_D.$ The surjective map $ \mathcal N \to \ii_D
$ induces a surjective map $ N = H^0_*(\mathcal N) \to I_D $
because $ \pd(\mathcal P) = s-2 $ implies that $ H^1_*(\mathcal P)
= 0 $ (see Lemma \ref{lemma1}). Hence, we have a well defines map
$ H_1 \to N $ given on the canonical bases of $ H_1 $ and extended
by linearity. So, there exists a map $ \varphi: \mathcal{H}_1 \to
\mathcal N.$ It is straightforward to check that $ \varphi $ maps
the kernel of $ \sigma_1 $ to the image of $ \mathcal P,$ and so $
\varphi $ induces a map $ \psi: \mathcal K \to \mathcal{P}.$ At
the end, there exists a commutative diagram
$$ \commdiag{0 \to &
\mathcal{K} & \mapright\lft{j} & \mathcal{H}_1 & \mapright & \ii_D
& \to 0 \cr & \mapdown\lft{\psi} & & \mapdown\lft{\varphi} & &
\bivline & \cr 0 \to & \mathcal{P} & \mapright & \mathcal{N} &
\mapright & \ii_D & \to 0 } $$ where the last map is the identity
of $ \ii_D.$ From the universal property of the push--out (see
\cite{north}, Ch. 3, Theorem 11 for the definition and the
properties of the push--out), it follows that $ \mathcal N $ is
the push--out of $ \mathcal H_1 $ and $ \mathcal P $ as claimed.

(ii) As soon as we fix a
map $ \psi: \mathcal K \to \mathcal P,$ we can construct the same
commutative diagram we considered in the first part of the proof.
In more detail, let $ q: \mathcal{K} \to \mathcal{H}_1 \oplus
\mathcal{P} $ be defined as $ j $ on the first summand and as $
-\psi $ on the second one. Then, $ \mathcal{N} = \mathcal{H}_1
\oplus \mathcal{P} / \mbox{im}(q).$ The sheaf $ \mathcal{N} $ is
torsion--free of rank $ \rk(\mathcal{P}) + 1.$ The second row of
commutative diagram above gives the short exact sequence $ 0 \to P
\to N \to I_D \to 0 $ because $ H^1_*(\mathcal P) = 0 $ by Lemma
\ref{lemma1}. Hence, $ \pd(\mathcal N) \leq \pd(\ii_D) =
\pd(\mathcal P) + 1,$ and the proof is complete.
\end{proof}

\begin{remark} \rm If $ \psi = 0,$ then $ \mathcal N = \mathcal P \oplus \ii_D,$
and the sequence is not interesting. On the other hand, if $ \psi
$ is an isomorphism, then $ \mathcal N \cong \mathcal H_1 $ and
once again we get nothing new.
\end{remark}

Summarizing the above discussed results, we have that if we start
from two sheaves $ \mathcal N $ and $ \mathcal P $ satisfying our
hypotheses, we can construct codimension $ s $ ACM schemes, and
conversely, given a codimension $ s $ ACM scheme $ D $ and a
torsion--free sheaf $ \mathcal P,$ we can construct a sheaf $
\mathcal{N} $ fulfilling the conditions we ask.

Starting from two given torsion--free sheaves $ \mathcal N $ and $
\mathcal P,$ there are constrains on the ACM schemes we can
obtain.
\begin{proposition} \label{constrains} In the same hypotheses as
Theorem \ref{NtoD}, let $ D \subset \pp^r $ be a codimension $ s $
ACM closed scheme whose ideal sheaf fits into a short exact
sequence
$$ 0 \to \mathcal{P} \to \mathcal{N} \to \ii_D(k) \to 0 $$ for
some $ k \in \zz.$ Then, the minimal number of generators of $ I_D
$ is not larger than $ \rk(\mathcal{F}_1) $ while the free modules
$ H_i $ that appear in the minimal free resolution of $ I_D(k) $
are direct summands of $ F_i \oplus G_{i-1}.$
\end{proposition}

\begin{proof} We constructed a free resolution of $ I_D(k) $ in Remark
\ref{rem3.3}. The minimal free resolution of $ I_D $ can be
obtained from this last one by cancelling suitable summands.
\end{proof}

As a consequence of the hypotheses (\ref{H-hyp}), to construct ACM
schemes of codimension $ s \geq 3,$ we have to consider a
torsion--free sheaf $ \mathcal P $ satisfying $ \pd(\mathcal P)
> 0,$ that is to say, $ \mathcal P $ non--dissoci\'{e}. On the other hand, if the
codimension of $ D $ is $ 2,$ then $ \mathcal P $ is dissoci\'{e}.
In this case, we have a more geometric interpretation of the
construction, and it can be compared with Serre's construction
(Hartshorne's one, respectively)  when $ \mathcal N $ is a rank $
2 $ vector bundle (reflexive sheaf, respectively).
\begin{proposition} \label{cod2} Let $ D \subset \pp^r $ be a codimension $ 2 $
ACM closed scheme, and let $ c $ be an integer such that $ H^0(D,
\omega_D(c)) \not= 0.$ Then, for every non--zero $ \xi \in H^0(D,
\omega_D(c)) $ we can construct a short non-split exact sequence
$$ 0 \to \oo_{\pp^r}(c-r-1) \to \mathcal{N} \to \ii_D \to 0 $$
with $ \mathcal{N} $ torsion--free, of rank $ 2 $ and $
\pd(\mathcal N) \leq 1.$
\end{proposition}

\begin{proof} By Serre's duality for $ \pp^r $ (\cite{hart}, Ch. III,
Theorem 7.1), we get $ \ext^1(\ii_D, \oo_{\pp^r}(c-r-1)) \cong
H^{r-1}(\pp^r, \ii_D(-c))'.$ From the inclusion $ D
\hookrightarrow \pp^r,$ we get $ H^{r-1}(\pp^r, \ii_D(-c))' \cong
H^{r-2}(D, \oo_D(-c))',$ and again by Serre's duality on $ D $
(\cite{hart}, Ch. III, Theorem 7.6 and Proposition 6.3(c)), we
have the further isomorphisms $ \ext^1(\ii_D, \oo_{\pp^r}(c-r-1))
\cong \Hom(\oo_D(-c), \omega_D) \cong H^0(D, \omega_D(c)).$ Hence,
every non--zero $ \xi \in H^0(D, \omega_D(c)) $ can be thought to
as an extension of $ \ii_D $ with $ \oo_{\pp^r}(c-r-1) $ and so as
a non--split short exact sequence $$ 0 \to \oo_{\pp^r}(c-r-1) \to
\mathcal{N} \to \ii_D \to 0.$$ The sheaf $ \mathcal{N} $ has rank
$ 2,$ and it is torsion--free. Moreover, if $$ 0 \to \mathcal{H}_2
\stackrel{\varphi}{\longrightarrow} \mathcal{H}_1 \to \ii_D \to 0
$$ is the minimal dissoci\'{e} resolution of $ \ii_D,$ there is a
natural surjection $ \Hom(\mathcal{H}_2, \oo_{\pp^r}(c-r-1)) \to
\ext^1(\ii_D, \oo_{\pp^r}(c-r-1)),$ and so there exists a map $
\psi: \mathcal{H}_2 \to \oo_{\pp^r}(c-r-1) $ that does not factor
through $ \varphi: \mathcal{H}_2 \to \mathcal{H}_1 $ whose image
in $ \ext^1(\ii_D, \oo_{\pp^r}(c-r-1)) $ is equal to $ \xi.$ By
using standard results from homological algebra, we get that $
\mathcal{N} $ is the push--out of $ \mathcal{H}_1 $ and $
\oo_{\pp^r}(c-r-1) $ via $ (\varphi, -\psi).$ Hence, the
resolution of $ \mathcal{N} $ with dissoci\'{e} sheaves is $$ 0
\to \mathcal{H}_2 \stackrel{(\varphi,-\psi)}{\longrightarrow}
\mathcal{H}_1 \oplus \oo_{\pp^r}(c-r-1) \to \mathcal{N} \to 0 $$
and so $ \mathcal{N} $ has projective dimension less than or equal
to $ 1.$
\end{proof}

\begin{remark} \rm From the proof of the previous Proposition,
we get that $ \pd(\mathcal N)=0 $ if and only if $ \mathcal{H}_2 =
\oo_{\pp^r}(c-r-1), $ i.e. $ D $ is a complete intersection
scheme.
\end{remark}

\begin{remark} \rm We can easily modify the proof to get sheaves $ \mathcal{N} $
of larger rank: it is enough to consider $ c_1, \dots, c_n \in \zz
$ such that $ H^0(D, \omega_D(c_i)) \not= 0 $ for at least a $
c_i.$ As in the proof of the previous Proposition, $
\oplus_{i=1}^n H^0(D, \omega_D(c_i)) \cong \ext^1(\ii_D,
\oplus_{i=1}^n \oo_{\pp^r}(c_i-r-1)) $ and so a non--zero element
$ \xi \in \oplus_{i=1}^n H^0(D, \omega_D(c_i)) $ can be considered
as an extension of $ \ii_D $ with $ \mathcal P = \oplus_{i=1}^n
\oo_{\pp^r}(c_i-r-1),$ and we can construct $ \mathcal N $ as in
the proof.
\end{remark}

\begin{remark} \rm In comparing Proposition \ref{cod2} with Serre's and Hartshorne's
constructions mentioned above, it is evident that the hypothesis
on $ \mathcal N $ strongly affects the properties of the
constructed scheme. For example, when $ \mathcal N $ is a rank $ 2
$ reflexive sheaf, as in Hartshorne' s setting, the associated
schemes are generically locally complete intersection. In fact,
the locus where the reflexive sheaf $ \mathcal N $ is not locally
free has codimension $ \geq 3 $ (\cite{hart-2}, Corollary 1.4 and
Theorem 4.1 for the case of curves in $ \pp^3 $). The properties
of the associated schemes show that the constructions are not the
same one. In fact, following Proposition \ref{cod2}, it is
possible to construct ACM schemes which are locally complete
intersection at no point, while if $ \mathcal N $ is reflexive and
$ D $ is the associated scheme, the locus of the points of $ D $
where $ D $ is not locally complete intersection has codimension $
\geq 1 $ in $ D.$ On the other hand, all the schemes constructed
via Proposition \ref{cod2} are ACM, while the ones associated to
reflexive sheaves can have non--zero cohomology.
\end{remark}

Now, we show how to construct ACM codimension 2 schemes which
contain the first infinitesimal neighborhood of another ACM
codimension 2 scheme. They are candidates to have no points at
which the scheme is locally complete intersection.

\begin{proposition} Let $ Y $ be an ACM codimension 2 scheme
and let  $ \mathcal{N} = \ii_Y \oplus \ii_Y.$ Then, every
codimension $ 2 $ ACM scheme $ D $ we obtain from the construction
above contains the first infinitesimal neighborhood of $ Y.$

Moreover, $D$ is not locally complete intersection at any point of
$Y$.In particular it is not generically locally complete
intersection.
\end{proposition}

\begin{proof} For the first statement, it is enough to prove that $ I_D \subset I_Y^2.$

Let $ 0 \to \mathcal L_1 \stackrel{\varphi}{\longrightarrow}
\mathcal L_0 \to \mathcal I_Y \to 0 $ be the minimal dissoci\'{e}
resolution of $ \mathcal I_Y. $ Let $ \varphi $ be represented by
a matrix $ A. $ Hence, the maximal minors of $ A $ generate the
ideal $ I_Y.$

Let $ \mathcal P = \oo_{\pp^r}(-m) $ and let $ \gamma: \mathcal P
\to \mathcal N $ be a general map whose image is a maximal
subsheaf of $ \mathcal N. $ Let $ \gamma': \mathcal P \to \mathcal
L_0 \oplus \mathcal L_0 $ be a lifting of $ \gamma.$

The ideal $ I_D $ is generated by the maximal minors of the matrix
$$ M = \left( \begin{array}{ccc} A & O & C' \\ O & A & C''
\end{array} \right) $$ where the last column represents
$ \gamma'.$ Every maximal minor of $ M $ can be computed by
Laplace rule with respect to the last column, and so it is a
combination of the maximal minors of the block matrix $ \left(
\begin{array}{cc} A & O \\ O & A \end{array} \right),$ whose
maximal minors generate the ideal $ I_Y^2.$

Let now  $x \in Y$ and set $S: =\oo_{{\pp^r},x}$. We have an exact
sequence of $S$-modules
$$0 \to S \to \mathcal N_x \to \ii_{D,x} \to 0.$$
It is easy to see that $\mathcal N_x$ needs at least four
generators whence $\ii_{D,x}$ needs at least three generators.
Since $D$ has codimension $ 2 $ it cannot be a complete
intersection at $x$.
\end{proof}

\begin{remark} \rm The easiest case we can consider is when the scheme $ Y $ is the
complete intersection of two hypersurfaces. In this case, the
scheme defined by $ I_Y^2 $ is ACM of codimension $ 2 $ and it can
be obtained from the previous construction.
\end{remark}

A similar result holds both for the direct sum of $ s (\geq 2 ) $
copies of $ \ii_Y,$ and for non--trivial extensions of $ \ii_Y $
with itself or with twists of another ACM codimension $ 2 $ scheme
$ Z $, but we do not state them.

Now, we relate extensions associated to divisors that differ by
hypersurface sections.
\begin{proposition} Let $ D \subset \pp^r $ be a codimension $ 2 $
ACM scheme. Let $ \xi \in H^0(D, \omega_D(c)) $ and $ \xi' \in
H^0(D, \omega_D(c+d)) $ both non--zero, with $ d \geq 0,$ and let
$$ 0 \to \oo_{\pp^r}(c-r-1) \to \mathcal{N} \to \ii_D \to 0 $$ and
$$ 0 \to \oo_{\pp^r}(c+d-r-1) \to \mathcal{N}' \to \ii_D \to 0 $$
be the associated short exact sequences. Then, there exists a
degree $ d $ hypersurface $ S = V(f) $ that cuts $ D $ along a
codimension $ 3 $ subscheme such that $ \xi' = f \xi $ if, and
only if, there exists a short exact sequence $$ 0 \to \mathcal{N}
\to \mathcal{N}' \to \oo_S(c+d-r-1) \to 0 $$ that induces the
identity on $ \ii_D.$
\end{proposition}

\begin{proof} In the proof of previous Proposition, we constructed
the sheaf $ \mathcal{N} $ as push--out $$\commdiag{\mathcal{H}_2 &
\mapright\lft{\varphi} & \hspace{.5 cm} \mathcal{H}_1  \cr
\mapdown\lft{\psi} & & \hspace{.5 cm} \mapdown \cr
\oo_{\pp^r}(c-r-1) & \mapright & \hspace{.5 cm} \mathcal{N} }.$$

Assume that $ \xi' = f \xi.$ The section $ f \xi \in H^0(D,
\omega_D(c+d)) $ is the image of the map $ f \psi \in
\Hom(\mathcal{H}_2, \oo_{\pp^r}(c+d-r-1)) $ in $ \ext^1(\ii_D,
\oo_{\pp^r}(c+d-r-1)) $ and so the sheaf $ \mathcal{N}' $ is the
push--out of $ \mathcal{H}_1 $ and $ \oo_{\pp^r}(c+d-r-1) $ via $
\varphi $ and $ -f \psi.$ From the universal property of the
push--out (see \cite{lang}, pp. 62), we get the following map of
complexes $$ \commdiag{0 \to & \oo_{\pp^r}(c-r-1) & \mapright &
\mathcal{N} & \mapright & \ii_D & \to 0 \cr & \mapdown\lft{f} & &
\mapdown\lft{\varepsilon} & & \bivline & \cr 0 \to &
\oo_{\pp^r}(c+d-r-1) & \mapright & \mathcal{N}' & \mapright &
\ii_D & \to 0 } $$ and so $ \varepsilon $ is injective, and $
\mbox{coker}(\varepsilon) \cong \oo_S(c+d-r-1),$ as claimed.

Assume now that the short exact sequence $$ 0 \to \mathcal{N} \to
\mathcal{N}' \to \oo_S(c+d-r-1) \to 0 $$ induces the identity on $
\ii_D.$ Standard arguments allow us to lift $ \varepsilon $ to an
injective map $ \oo_{\pp^r}(c-r-1) \to \oo_{\pp^r}(c+d-r-1) $
whose cokernel is isomorphic to $ \oo_S(c+d-r-1).$ Hence, the map
is the multiplication by $ f,$ and $ \mathcal{N}' $ is the
push--out of $ \mathcal{H}_1 $ and $ \oo_{\pp^r}(c+d-r-1) $ via $
\varphi $ and $ f \psi.$ Hence, $ \xi' = f \xi,$ and the proof is
complete.
\end{proof}

\begin{remark} \rm Let $ \xi, \xi' \in H^0(D, \omega_D(c)).$ By applying the
previous Proposition, we get that $ \xi $ and $ \xi' $ are
linearly dependent if and only if the sheaves $ \mathcal N $ and $
\mathcal N' $ associated to them are isomorphic.
\end{remark}


\section{ACM schemes from ACM ones}

Let $ X \subset \pp^r $ be a codimension $ t $ ACM scheme. For
general choices, $ s (< t) $ hypersurfaces of large degree
containing $ X $ define a complete intersection codimension $ s $
ACM scheme containing $ X.$ In this section, we discuss the
related problem of finding an ACM codimension $ s $ closed scheme
$ D \subset \pp^r $ containing $ X.$ Of course, we make use of the
construction described in the previous section.

The main result is the following.
\begin{proposition} \label{XtoD} Let $ X $ be a codimension $ t $
ACM scheme in $ \pp^r $ with $ 3 \leq t \leq r $ and let
\begin{equation} \label{res_IX} 0 \to F_t
\stackrel{\delta_t}{\longrightarrow} F_{t-1} \to \dots \to F_2
\stackrel{\delta_2}{\longrightarrow} F_1
\stackrel{\delta_1}{\longrightarrow} I_X \to 0
\end{equation} be the minimal free resolution of the saturated
ideal that defines $ X.$ Let $ N = \ker(\delta_{t-s}) $ be the $
(t-s)$--th syzygy module of $ X,$ for some $ s \geq 2,$ and let $
P $ be a torsion--free $ R$--module of projective dimension $
s-2.$ Assume further that $ \mathcal N $ and $ \mathcal P $
satisfy the condition $ (H.3),$  and that there exists an
injective map $ \gamma: \mathcal{P} \to \mathcal{N} $ such that $
\gamma(\mathcal{P}) $ is a maximal subsheaf of $ \mathcal{N}.$
Then, for every ACM codimension $ s $ closed scheme $ D $
constructed as in Theorem \ref{NtoD} there is a short exact
sequence
$$ 0 \to
\mathcal{E}xt^{s-2}(\mathcal{P}, \omega_{\pp^r}) \to \omega_D(-k)
\to \omega_X \to 0. $$ Moreover, $ D $ contains $ X. $
\end{proposition}

\begin{proof} The $ R$--module $ N $ is torsion--free, and has
no free summand, because it is computed from the minimal free
resolution of $ I_X.$ Moreover,
$$ 0 \to F_t \stackrel{\delta_t}{\longrightarrow} F_{t-1}
\stackrel{\delta_{t-1}}{\longrightarrow} \dots
\stackrel{\delta_{t-s+2}}{\longrightarrow} F_{t-s+1} \to N \to 0
$$ is the minimal free resolution of $ N $ and so the projective
dimension of $ N $ is $ s-1.$ Hence, $ \mathcal N $ and $ \mathcal
P $ satisfy all the conditions (\ref{H-hyp}).

Hence, by Theorem \ref{NtoD} there exists a codimension $ s $ ACM
closed scheme $ D \subset \pp^r,$ and an integer $ k $ such that
$$ 0 \to \mathcal{P} \to \mathcal{N} \to \ii_D(k) \to 0
$$ is a short exact sequence. By applying $ \mathcal{H}om(-,
\omega_{\pp^r}) $ we get
\begin{equation*} \begin{split} \mathcal{E}xt^{s-2}(\mathcal{N},
\omega_{\pp^r}) & \to \mathcal{E}xt^{s-2}(\mathcal{P},
\omega_{\pp^r}) \to \mathcal{E}xt^{s-1}(\ii_D(k), \omega_{\pp^r})
\to \\ & \to \mathcal{E}xt^{s-1}(\mathcal{N}, \omega_{\pp^r}) \to
\mathcal{E}xt^{s-1}(\mathcal{P}, \omega_{\pp^r}). \end{split}
\end{equation*}

$ \mathcal{E}xt^{s-1}(\mathcal{P}, \omega_{\pp^r}) = 0 $ because $
\pd(\mathcal P) = s-2,$ while $ \mathcal{E}xt^{s-j}(\mathcal{N},
\omega_{\pp^r}) = \mathcal{E}xt^{t-j}(\ii_X, \omega_{\pp^r}) $ by
definition of $ N.$ Hence, $ \mathcal{E}xt^{s-1}(\mathcal{N},
\omega_{\pp^r}) = \omega_X,$ and $
\mathcal{E}xt^{s-2}(\mathcal{N}, \omega_{\pp^r}) = 0 $ because $ X
$ is ACM of codimension $ t $ (\cite{hart}, Ch. III, Proposition
7.5 and Theorem 7.1). Again by (\cite{hart}, Ch. III, Proposition
7.5), $ \mathcal{E}xt^{s-1}(\ii_D(k), \omega_{\pp^r}) =
\omega_D(-k).$ Summarizing the above arguments, the construction
induces a short exact sequence $$ 0 \to
\mathcal{E}xt^{s-2}(\mathcal{P}, \omega_{\pp^r}) \to \omega_D(-k)
\to \omega_X \to 0 $$ that relates the dualizing sheaves of $ X $
and $ D.$ In particular, we can think of $ \omega_X $ as a
quotient of $ \omega_D,$ up to a twist. The annihilator of $
\omega_X $ is $ \ii_X $ (see, \cite{ei}, Corollary 21.3), the one
of $ \omega_D $ is $ \ii_D,$ and so we get the last claim because
it is evident that the annihilator of $ \omega_D(-k) $ is
contained in the one of $ \omega_X.$
\end{proof}

The previous Proposition explains our motivation in studying the
exact sequences as (\ref{pnd}). In fact, we applied the
construction by M.Martin--Deschamps and D.Perrin to the first
syzygy module $ N $ of a zero--dimensional scheme $ X $ in $
\pp^3,$ i.e. $ \pd(N) = 1.$ The above mentioned construction
provides a free module $ P $ ($ \pd(P) = 0 $) and a general
injective map $ \gamma: P \to N $ whose cokernel is, up to a
twist, the ideal of a curve $ D $ (and so the codimension of $ D $
is $ 2 $). Hence, the hypotheses of Proposition \ref{XtoD} are
fulfilled and the curve $ D $ is ACM and contains $ X.$

We rephrase Proposition \ref{constrains} in the case $ \mathcal N
$ is the $ (t-s)$--syzygy sheaf of an ACM scheme $ X $ of
codimension $ t.$
\begin{corollary} \label{constrains1} Let $ X $ and $ D $ be
schemes as in Proposition \ref{XtoD}. Then, the Cohen--Macaulay
type of $ X $ is not greater than the one of $ D.$ In particular,
$ D $ is arithmetically Gorenstein if and only if $ X $ is such.
\end{corollary}

\begin{proof} The minimal dissoci\'{e} resolution of $ \mathcal N
$ agrees with the one of $ \ii_X,$ and so $ F_t \oplus G_{s-1} $
appears in a free resolution of $ I_D(k),$ as it follows from
Remark \ref{rem3.3}. $ F_t $ cannot be cancelled because it maps
to $ F_{t-1} $ and the resolution of $ I_X $ is minimal, and so
the first claim follows. In particular, $ F_t $ is equal to the
last free module in a minimal free resolution of $ I_D(k) $ if and
only if $ \gamma_{s-1} : G_{s-1} \to F_{t-1} $ is
split--injective, where $ \gamma_{s-1} $ is induced from $ \gamma:
P \to N.$ The second statement is straightforward.
\end{proof}

For example, if $ X \subset \pp^3 $ is a set of $ 5 $ general
points, it is arithmetically Gorenstein with Pfaffian resolution
$$ 0 \to R(-5) \to R^5(-3) \to R^5(-2) \to I_X \to 0.$$ By
applying the previous construction with $ P = R^3(-3),$ we get
that $ k = -1 $ and the minimal free resolution of $ I_D $ is
$$ 0 \to R(-5) \to R^2(-3) \to I_D(-1) \to 0,$$ so $ D $ is a
complete intersection curve in $ \pp^3.$

\begin{remark} \label{4-points} Among the ACM closed schemes
$ D $ constructed in Proposition \ref{XtoD} we might not find the
ones of minimal degree containing $ X.$ For example, let $ X
\subset \pp^3 $ be the degree $ 4 $ reduced scheme consisting of
the vertices of the unit tetrahedron. With an easy computation, we
get that $ I_X $ is generated by $ xy, xz, xw, yz, yw, zw,$ and
its minimal free resolution is $$ 0 \to R^3(-4) \to R^8(-3) \to
R^6(-2) \to I_X \to 0.$$ An  ACM curve $ C $  of minimal degree
containing $ X $ is the union of the three lines $ V(x,y), V(y,z),
V(z,w).$ The minimal free resolution $I_C$ is $$ 0 \to R^2(-3) \to
R^3(-2) \to I_C \to 0.$$ It follows that $ C $ cannot be obtained
from Proposition \ref{XtoD} because the Cohen--Macaulay types of $
X $ and $ C $ are $ 3 $ and $ 2,$ respectively, and this is not
possible by Corollary \ref{constrains1}.
\end{remark}

\begin{example} \rm In this example, we construct two ACM curves
with different Cohen--Macaulay types starting from the same $ X.$

Let $ r = 3 $ and let $ X $ be a set of four general points in a
plane. Of course, $ I_X $ is the complete intersection of a linear
form and two quadratic forms, and so its minimal free resolution
is $$ 0 \to R(-5) \to R^2(-3) \oplus R(-4) \to R(-1) \oplus
R^2(-2) \to I_X \to 0.$$

If we choose $ P = R(-3),$ we get a complete intersection curve $
D $ whose minimal free resolution is $$ 0 \to R(-5) \to R(-3)
\oplus R(-4) \to I_D(-2) \to 0.$$

On the other hand, if we choose $ P = R(-5),$ we get an ACM curve
$ E $ whose minimal free resolution is $$ 0 \to R^2(-5) \to
R^2(-3) \oplus R(-4) \to I_E \to 0.$$

Both curves are constructed by choosing a general injective map
from $ P $ to $ R^2(-3) \oplus R(-4).$
\end{example}

Summarizing the obtained results, we proved that it is possible to
construct a codimension $ s $ ACM closed scheme $ D $ containing a
given codimension $ t $ ACM scheme $ X $ as soon as $ s < t.$ Some
of the restrictions are: the number of minimal generators of $ I_D
$ cannot be larger than the number of minimal generators of the $
R$--module $ N $ we used in the construction, and the last free
module in a minimal free resolution of $ I_X $ is a direct summand
of the last free module in a minimal free resolution of $ I_D(k).$
A consequence of the restrictions is that there are ACM schemes
containing $ X $ that cannot be constructed as explained in
Proposition \ref{XtoD} (e.g., see Remark \ref{4-points}).

The last result we present in this section allows us to
reconstruct an ACM scheme $ D $ from a subscheme $ X $ of $ D $
obtained by intersecting $ D $ with a complete intersection $ S.$

\begin{proposition} Let $ D \subset \pp^r $ be a codimension $ s $ ACM scheme with
minimal free resolution $$ 0 \to H_s
\stackrel{\varepsilon_s}{\longrightarrow} H_{s-1}
\stackrel{\varepsilon_{s-1}}{\longrightarrow} \dots
\stackrel{\varepsilon_2}{\longrightarrow} H_1 \to I_D \to 0,$$ and
let $ S = V(f_1, \dots, f_t) $ be a codimension $ t $ complete
intersection scheme that cuts $ D $ along a codimension $ s+t \leq
r $ scheme $ X.$ Then, $ D $ can be constructed from $ X $ as
explained in Proposition \ref{XtoD}.
\end{proposition}

\begin{proof} Let $ F = \oplus_{i=1}^t R(-deg(f_i)).$ Then, the
minimal free resolution of $ I_S $ is given by the Koszul complex
$$ 0 \to \wedge^t F \stackrel{\varphi_t}{\longrightarrow}
\wedge^{t-1} F \stackrel{\varphi_{t-1}}{\longrightarrow} \dots
\stackrel{\varphi_2}{\longrightarrow} F
\stackrel{\varphi_1}{\longrightarrow} I_S \to 0 $$ where $
\varphi_i = \wedge^i \varphi $ and $ \varphi:F \to R $ is defined
as $ \varphi(e_i) = f_i $ for each $ i = 1, \dots, t,$ where $
e_1, \dots, e_t $ is the canonical basis of $ F.$

Let $ X = D \cap S,$ and let $ I_X \subset R $ be its saturated
ideal. It is easy to prove that a free resolution of $ I_X $ can
be constructed as tensor product of the resolutions of $ I_D $ and
$ I_S $ (for the definition of the tensor product of complexes see
Section 17.3 in \cite{ei}). Hence, it is equal to $$ 0 \to G_{s+t}
\to G_{s+t-1} \to \dots \to G_1 \to I_X \to 0 $$ where $$ G_h =
\bigoplus_{i+j=h,\ i, j \geq 0} H_i \otimes \wedge^j F
$$ for $ h = 1, \dots, s+t,$ and the map $ \delta_h : G_h \to
G_{h-1} $ restricted to $ H_i \otimes \wedge^j F \to (H_{i-1}
\otimes \wedge^j F) \oplus (H_i \otimes \wedge^{j-1} F) $ is
defined as $$ \delta_i = \left(
\begin{array}{c} \varepsilon_i \otimes 1 \\ (-1)^i\ 1 \otimes
\varphi_j \end{array} \right).$$ In particular, $ X $ is ACM of
codimension $ s+t.$

Let $ N $ be the kernel of $ \delta_t,$ and so a resolution of $ N
$ is equal to $$ 0 \to G_{s+t} \to \dots \to G_{t+1} \to N \to
0.$$ Moreover, $ N $ is torsion--free.

Now, let $ G'_{t+j} = (H_{j+1} \otimes \wedge^{t-1} F) \oplus
\dots \oplus (H_s \otimes \wedge^{t+j-s} F) $ for $ j = 1, \dots,
s-1.$ Of course, $ G_{t+j} = (H_j \otimes \wedge^t F) \oplus
G'_{t+j}.$ Let $ \Delta_{t+j} : G'_{t+j} \to G'_{t+j-1} $ be the
restriction of $ \delta_{t+j} $ to $ G'_{t+j},$ and let $ P =
\mbox{coker}(\Delta_{t+2}).$ A free resolution of $ P $ is $$ 0
\to G'_{s+t-1} \to G'_{s+t-2} \to \dots \to G'_{t+1} \to P \to
0.$$ In fact, it is easy to prove that it is a complex.
Furthermore,  it is exact, because it is a sub--complex of the
resolution of $ I_X.$ It is obvious that the inclusion $ G'_{t+j}
\to G_{t+j} $ for $ j \geq 1,$ induces an inclusion $ P \to N.$
The resolution of the cokernel is $$ 0 \to H_s \otimes \wedge^t F
\to H_{s-1} \otimes \wedge^t F \to \dots \to H_1 \otimes \wedge^t
F \to N/P \to 0.$$ But $ \wedge^t F \cong R(-\sum_{i=1}^t
\deg(f_i)) $ and the maps are $ \varepsilon_i \otimes 1.$ Hence, $
\mathcal{N}/\mathcal{P} \cong \ii_D(k) $ where $ k = -\sum_{i=1}^t
\deg(f_i).$ In particular, from Proposition 2.9 it follows that $
\mathcal{P} $ is a maximal sub--sheaf of $ \mathcal{N},$ and so
the claim is proved.
\end{proof}

\begin{remark} \rm  In the previous theorem, suppose $ D $ is a complete intersection.
Then, $ X $ is a complete intersection too and $ I_X $ is
generated by a regular sequence obtained by taking all the
generators of $ I_D $ and $ I_S. $

Reversing this observation, we consider a complete intersection
scheme $ X $ generated by a regular sequence of forms $ (f_0,
\dots, f_i), $ with $ i \leq r.$ Starting from $ X $ we can obtain
all the schemes $ D $ generated by a subset of generators of $ X.
$ In particular, if we take $ i=r, $ and $ f_j=x_j, j=0, \dots r,
$ the $(r-1)$--syzygy sheaf $ \mathcal N $ involved in the
construction is a twist of the tangent sheaf $ T_{\pp^r}.$
\end{remark}




\begin{thebibliography}{99}

\bibitem{bs} M.P. Brodmann, R.Y. Sharp, \em Local Cohomology: an
algebraic introduction with geometric applications, \rm Cambridge
Univ. Press (1998).

\bibitem{ei} D. Eisenbud, \em Commutative Algebra, \bf GTM 150, \rm Springer
Verlag (1994).

\bibitem{gr} M. Green, \em Koszul cohomology and geometry, \rm
in Lectures on Riemann Surfaces, ed. M. Cornalba, World Scientific
Press (1989).

\bibitem{hart-1974} R. Hartshorne, \em Varieties of small codimension
in projective space, \rm Bull. AMS \bf 80 \rm (1974), 1017--1032.

\bibitem{hart} R. Hartshorne, \em Algebraic Geometry, \bf GTM 52,\rm
Springer Verlag (1977).

\bibitem{hart-2} R. Hartshorne, \em Stable reflexive sheaves, \rm
Math. Ann. \bf 254 \rm (1980), 121--176.

\bibitem{lang} S. Lang, \em Algebra, \bf GTM 211, \rm Springer Verlag (2002).

\bibitem{mdp} M. Martin-Deschamps, D. Perrin, \em Sur la
classification des courbes gauches, \rm Ast\'{e}risque \bf
184-185, \rm (1990).

\bibitem{matsumura} H. Matsumura, \em Commutative Algebra, \rm
Math. Lecture Note Series, The Benjamin/Cummming Pub. Co., (1980).

\bibitem{m} J.C. Migliore, \em Introduction to Liaison Theory and
Deficiency Modules, \rm Progress in Math., \bf Vol. 165, \rm
Birkh$\ddot{a}$user (1998).

\bibitem{north} D.G. Northcott, \em A first course of homological
algebra, \rm Cambridge University Press (1973).

\bibitem{okonek} C. Okonek, M. Schneider, H. Spindler, \em Vector
Bundles on Complex Projective Spaces, \rm Progress in Math., \bf
Vol. 3, \rm Birkh$\ddot{a}$user (1980).

\bibitem{serre-1960} J.P. Serre, \em Sur les modules projectifs,
\rm S\'{e}m. Dubreil-Pisot, exp. 2, (1960/61), 1--16.

\end{thebibliography}
\end{document}